\documentclass[11pt,a4paper]{amsart}
\usepackage[T1]{fontenc}
\usepackage{lmodern}
\usepackage[a4paper,margin=28mm]{geometry}
\usepackage{amsmath,amssymb,amsthm,mathtools}
\usepackage{microtype,enumitem,booktabs,longtable,array}
\usepackage{xcolor,tikz}
\usepackage{hyperref}
\hypersetup{colorlinks=true,linkcolor=blue!45!black,citecolor=blue!45!black,
  urlcolor=blue!45!black,pdftitle={Square-difference factorization: structure and torsion},
  pdfauthor={Malik Tusif Ahmed, Hafnaoui Hidouri, Ahmed Hamed, and Hwankoo Kim}}
\setlist[enumerate]{label=\textup{(\arabic*)},leftmargin=*,itemsep=3pt}
\numberwithin{equation}{section}
\newtheorem{theorem}{Theorem}[section]
\newtheorem{proposition}[theorem]{Proposition}
\newtheorem{lemma}[theorem]{Lemma}
\newtheorem{corollary}[theorem]{Corollary}
\theoremstyle{definition}

\newtheorem{example}[theorem]{Example}
\theoremstyle{remark}
\newtheorem{remark}[theorem]{Remark}
\DeclareMathOperator{\Spec}{Spec}
\DeclareMathOperator{\Max}{Max}

\DeclareMathOperator{\Ann}{Ann}
\DeclareMathOperator{\Reg}{Reg}

\DeclareMathOperator{\Pic}{Pic}

\newcommand{\Z}{\mathbb Z}
\newcommand{\F}{\mathbb F}

\title[Square-difference factorization: structure and torsion]
{Square-difference factorization:\\
ring structure, torsion splitting, and ideal classes}
\author[M.~T. Ahmed]{Malik Tusif Ahmed}
\address{(Ahmed) Department of Mathematics and Statistics, Hazara University,
Mansehra, Pakistan}
\email{tusifahmed@hu.edu.pk, maliktusifahmed@gmail.com}

\author[A.~Hamed]{Ahmed Hamed}
\address{(Hamed) Department of Mathematics, Faculty of Sciences,
Monastir, Tunisia}
\email{hamed.ahmed@hotmail.fr}

\author[H.~Hidouri]{Hafnaoui Hidouri}
\address{(Hidouri) Department of Mathematics, Faculty of Sciences,
Monastir, Tunisia}
\email{hafnaouihidouri6@gmail.com}

\author[H.~Kim]{Hwankoo Kim$^\dag$}
\address{(Kim) Department of Computer Engineering, Hoseo University, Republic of Korea}
\email{hkkim@hoseo.edu}

\date{\today}
\thanks{$\dag$: Corresponding author}

\begin{document}

\begin{abstract}
We introduce and study commutative rings whose regular ideals, or all
ideals, factor into square-difference factor absorbing ideals. We establish
permanence under finite products, suitable localizations and quotients,
and coefficient descent for monoid algebras. For an SP-domain, the
regular factorization property is equivalent to the localization obtained
by inverting $2$ being zero or Dedekind; equivalently, the maximal ideals
avoiding $2$ are invertible and form a discrete subspace in the inverse
topology. Over an arbitrary domain, we characterize unique functorial
splitting of localization on all torsion modules by a condition excluding
nonzero prime chains across the vanishing set of the localized element.
Complementary localizations of one-dimensional Pr\"ufer domains give a
decomposition of invertible ideal groups and an exact Picard-group sequence.
For an SP-domain whose Picard group has finite exponent $m$, every
irreducible element has local valuation at most $m$.
We also give regular-ideal criteria for arbitrary idealizations and a
complete strong-factorization criterion over von Neumann regular bases.
The zero ideal is treated separately throughout, and explicit examples
show why unrestricted quotient and localization assertions fail.
\end{abstract}

\maketitle

\section{Introduction}

The passage from prime factorization to radical factorization replaces
Dedekind domains by a larger class of almost Dedekind domains.
Vaughan and Yeagy introduced SP-domains in this form
\cite{VaughanYeagy1978}; later work connected their ideals with
maximal spectra, valuation functions, and ordered ideal groups
\cite{Olberding2005, McGovern2007, OlberdingReinhart2019Lattices, Spirito2023}.
For rings with zero divisors, factorization of regular ideals and
factorization of all ideals give distinct theories, developed by
Ahmed and Dumitrescu \cite{AhmedDumitrescu2017}.

A proper ideal $I$ is \emph{square-difference factor absorbing}, or
\emph{sdf-absorbing}, if $a^2-b^2\in I$ for nonzero elements $a,b$
implies $a+b\in I$ or $a-b\in I$. Anderson, Badawi, and Coykendall
introduced this condition \cite{ABC2026}, and Hizem characterized
it in terms of minimal primes \cite{Hizem2026}.
We introduce \emph{SAF-rings} and \emph{SSAF-rings} by requiring
such factorizations for regular ideals and for all ideals, respectively.
Their basic structure involves both the distinction between these two
classes of ideals and the residue characteristics of the ring.

The identity $(a-b)(a+b)=a^2-b^2$ suggests an interpolation between
prime and radical factorization, but the interpolation depends on the
element $2$. In characteristic $2$, radical ideals satisfy the sdf
condition. If $2$ is invertible, nonzero sdf ideals are prime.
The main issue is therefore how to assemble ideal factorizations when
residue fields of characteristic $2$ coexist with other residue fields.
The localization $D[1/2]$ and the complementary maximal supports give
a precise answer for domains.

The principal contributions of this paper are the following.
\begin{enumerate}
\item Proposition~\ref{prop:permanence} establishes the product,
localization, and quotient properties of SAF and SSAF, with explicit
regularity hypotheses. Corollaries~\ref{cor:amalgamation-descent}
and~\ref{cor:monoid-descent} give descent for amalgamations and monoid
algebras. Proposition~\ref{prop:classical-comparison} locates these
classes among radical and prime ideal factorization classes.
\item Theorem~\ref{thm:domain-criterion} characterizes SAF among
SP-domains by the Dedekind property of $D[1/2]$, finite character away
from $2$, and discreteness in the inverse topology.
Corollary~\ref{cor:saf-overrings} proves permanence under overrings.
\item Theorem~\ref{thm:torsion-prime-chain-criterion} identifies the
prime-chain obstruction to unique functorial splitting on all torsion
modules over an arbitrary domain. The SP case admits the explicit
cyclic decomposition in Theorem~\ref{thm:cyclic-torsion}.
\item Theorem~\ref{thm:complementary-overrings} separates invertible
ideals over two complementary localizations and identifies the
obstruction to separating principal ideals by an exact sequence of
Picard groups. Proposition~\ref{prop:picard-exponent-irreducibles}
improves the SP-domain specialization of an irreducible-divisor bound
of Reinhart from $2m$ to $m$, where $m$ is the Picard-group exponent.
\item Proposition~\ref{prop:regular-idealization} characterizes the
regular factorization property of an arbitrary idealization.
Theorem~\ref{thm:vnr-idealization-ssaf} gives the strong property over
von Neumann regular bases without a characteristic-$2$ hypothesis,
and Theorem~\ref{thm:comparison-vnr-idealizations} compares it with
absorbing ideal factorization.
\end{enumerate}

The classification of individual sdf ideals is established material:
we explicitly use the results of Anderson--Badawi--Coykendall and Hizem.
The Marot-ring hierarchy and the Noetherian comparison use the
radical-factorization theory of Ahmed and Dumitrescu
\cite{AhmedDumitrescu2017}; the basic transfer properties are proved
directly for the sdf condition. In particular, the passage to a quotient
is justified by containment of the kernel in every factor.
The idempotent mechanism underlying torsion splitting is related to
\cite{KimKoc2026Strong}; our prime-chain formulation specifies exactly
when that mechanism works simultaneously on all torsion modules.
The comparisons with idealization theory use
\cite{AndersonWinders2009, DumitrescuMahdouZahir2021}, and those with
absorbing factorization use
\cite{AhmedDumitrescuKhadam2020, ElKhalfiLaarabiKoc2026, MukhtarAhmedDumitrescu2018}.

Section~\ref{sec:sdf-background} establishes the conventions and
permanence results, including the exceptional zero ideal.
Section~\ref{sec:domains} proves the domain criterion and gives
strict examples. Section~\ref{sec:torsion} treats torsion modules.
Section~\ref{sec:complementary-overrings} studies complementary
overrings and ideal classes, and Section~\ref{sec:idealizations}
treats rings with zero divisors through idealizations.

\subsection{Conventions and terminology}

All rings are commutative with identity, all homomorphisms preserve
identity, and modules are unital. Rings being classified are nonzero;
the zero ring is allowed as a localization or an omitted product factor.
Every ideal used as a factor is proper. The unit ideal is an empty
product. A regular element is a non-zero-divisor, and an ideal is regular
if it contains a regular element. We write $\Reg(R)$ for the regular
elements and $T(R)=\Reg(R)^{-1}R$ for the total quotient ring.
A ring is \emph{Marot} if its regular ideals are generated by regular
elements; it is \emph{regular-Noetherian} if its regular ideals are
finitely generated.

An \emph{SP-ring} is a ring whose regular proper ideals are finite
products of radical ideals; an \emph{SSP-ring} has this property for all
proper ideals. Replacing radical factors by sdf factors gives SAF and
SSAF, respectively. For domains the two sdf properties coincide,
because zero is prime. A \emph{Dedekind ring} means a ring whose regular
proper ideals are finite products of prime ideals; a \emph{ZPI-ring}
requires this for all proper ideals. A Dedekind domain is Noetherian,
integrally closed, and of dimension at most one; fields are included.
A nonfield domain is \emph{almost Dedekind} if its localizations at
maximal ideals are DVRs, and is \emph{Pr\"ufer} if its nonzero finitely
generated ideals are invertible. A B\'ezout domain has all finitely
generated ideals principal.

For a domain $D$, put $D_2=D[1/2]$ and
$V(I)=\{M\in\Max(D):I\subseteq M\}$. If $D$ is almost Dedekind and
$I$ is a nonzero ideal, define $\nu_I(M)$ by
$ID_M=(MD_M)^{\nu_I(M)}$. We use the standard facts that a nonfield SP-domain
is almost Dedekind and that an almost Dedekind domain is SP precisely
when each maximal ideal contains a nonzero finitely generated radical
ideal \cite{VaughanYeagy1978,Olberding2005,McGovern2007}.
These facts will only be applied after their hypotheses are checked.
Every total quotient ring is SAF, since it has no regular proper ideal.
The principal ideal domains $\Z$ and $k[X]$, for a field $k$, are
SAF and SSAF because their ideals have prime ideal factorizations.
Every Boolean ring is SSAF: its ideals are radical, and in characteristic
$2$ the sdf test follows from $(a-b)^2=a^2-b^2$.

\section{Sdf ideals and permanence}\label{sec:sdf-background}

The original definition uses nonzero testing elements. Consequently,
radicality requires care at the zero ideal.
Anderson--Badawi--Coykendall \cite[Theorem~2.2]{ABC2026} assume that the ideal is nonzero.
Their Remark~2.3 gives the counterexample $0\subseteq\mathbb Z/4\mathbb Z$.
We recall the exact classification due to Hizem. The proof is included
for self-containedness and will be used to justify the implication
from SSAF to SSP under the original nonzero-element convention.

\begin{proposition}[Hizem {\cite[Theorem~11]{Hizem2026}}]\label{prop:exceptional-zero}
Let $R$ be a nonreduced ring.  Then its zero ideal is sdf-absorbing if and
only if $R$ is isomorphic to one of
\[
  \mathbb Z/4\mathbb Z,\qquad \mathbb Z/9\mathbb Z,\qquad
  \mathbb F_3[\varepsilon]/(\varepsilon^2).
\]
In particular, if $0$ is sdf-absorbing in a nonreduced ring, then
$\sqrt{0}^{\,2}=0$.
\end{proposition}

\begin{proof}
Choose $0\ne x\in R$ with $x^2=0$, by taking a suitable power of a nonzero
nilpotent.  Every square-zero element is $0$, $x$, or $-x$, since its square
can be compared with $x^2$.  Applying this observation to $2x$ shows that
$x$ has additive order either $2$ or $3$.

Suppose first that $2x=0$.  The nonzero units $1$ and $1+x$ have equal
squares, and their difference is nonzero.  Hence $2+x=0$.  It follows that
$x=2$ and $\operatorname{char}(R)=4$.  If
$r\in\operatorname{Ann}_R(x)\setminus\{0,x\}$, the elements $r$ and $r+x$
are nonzero and have equal squares; their difference is $-x\ne0$, and their
sum is $2r+x=x\ne0$.  This is impossible.  Thus
$\operatorname{Ann}_R(x)=\{0,x\}$.  Also $xR=\{0,x\}$, since every element
of $xR$ is square-zero.  Multiplication by $x$ now shows that $R$ has four
elements, and its characteristic gives $R\cong\mathbb Z/4\mathbb Z$.

Suppose that $x$ has additive order $3$.  If
$r\in\operatorname{Ann}_R(x)\setminus\{0,x,-x\}$, both pairs
$(r,r+x)$ and $(r,r-x)$ consist of nonzero elements with equal squares.
Their differences are nonzero, so sdf-absorption gives both $2r=-x$ and
$2r=x$, a contradiction.  Therefore
$\operatorname{Ann}_R(x)=xR=\{0,x,-x\}$.  Multiplication by $x$ shows that
$R$ has nine elements.  If its characteristic is $9$, it is
$\mathbb Z/9\mathbb Z$.  If its characteristic is $3$, the elements $1,x$
form an $\mathbb F_3$-basis and $x^2=0$, yielding
$\mathbb F_3[\varepsilon]/(\varepsilon^2)$.

For the converse, in $\mathbb Z/4\mathbb Z$ the fibers of squaring on nonzero elements are
$\{1,3\}$ and $\{2\}$.  In $\mathbb Z/9\mathbb Z$ every nonzero square
fiber consists of a pair of opposite elements, and the nonzero elements
with square zero are $3,6$.  In $\mathbb F_3[\varepsilon]/(\varepsilon^2)$,
write an element as $a+b\varepsilon$.  If $a\ne0$, its square determines
the element up to its sign; if $a=0$, the two nonzero possibilities are
opposites.  This proves all three converse assertions.
\end{proof}

\begin{corollary}\label{cor:ssaf-ssp}
Every SSAF-ring is an SSP-ring. Moreover, every ideal of an SSAF-ring
has a factorization into radical sdf-absorbing ideals.
\end{corollary}

\begin{proof}
Any factorization of a nonzero ideal uses only nonzero factors, and those
sdf-absorbing factors are radical.  A factorization of zero either already
has only nonzero, hence radical, factors, or it has a zero factor.  In the
latter case $0$ is sdf-absorbing.  If the ring is reduced, zero is radical;
otherwise the preceding proposition gives $0=\sqrt0\,\sqrt0$,
and $\sqrt0$ is prime in each of the three exceptional rings.
This also proves the assertion about radical sdf factors.
\end{proof}

\subsection{Prime support after inverting \texorpdfstring{$2$}{2}}

The minimal-prime equivalence in the next proposition is Hizem's
criterion \cite[Proposition~4 and Corollary~7]{Hizem2026};
compare also \cite[Theorem~4.9]{ABC2026}.
The localization and colon statements are equivalent formulations
of that known result. We include a proof that specifies the
zero-ideal convention.

\begin{proposition}\label{prop:sdf-localization}
Let $I$ be a nonzero proper ideal of $R$.  The following are equivalent:
\begin{enumerate}
\item $I$ is sdf-absorbing;
\item $I$ is radical and $(R/I)[1/2]$ is either zero or a domain;
\item $I$ is radical and $(I:2)$ is either $R$ or a prime ideal;
\item $I$ is radical and at most one minimal prime over $I$ does not
contain $2$.
\end{enumerate}
For a radical ideal $I$, the same equivalence holds without the nonzero
hypothesis.
\end{proposition}

\begin{proof}
For nonzero $I$, radicality follows by comparing $a^2$ with $i^2$ for a
fixed $0\ne i\in I$ whenever $a^2\in I$.  Moreover, for radical $I$ the
nonzero restrictions on the two testing elements can be dropped.
Consequently, after passing to $A=R/I$, we may assume that $A$ is reduced
and ask whether equality of squares implies equality up to sign.

For a reduced ring $A$, the kernel of $A\to A[1/2]$ is
$\operatorname{Ann}_A(2)$.  Indeed, $2^nx=0$ implies
$(2x)^n=0$, hence $2x=0$.  If $A[1/2]$ is zero or a domain and
$(a-b)(a+b)=0$, one of $a-b,a+b$ has image zero after localization.
Say $2(a-b)=0$.  Then
$(a-b)^2=(a-b)(a+b)-2b(a-b)=0$, and reducedness gives $a=b$.
The other case gives $a=-b$.

Conversely, the equality-of-squares property localizes. In fact, write
two elements with a common denominator as $a/s,b/s$. An equality of
their squares gives $u(a^2-b^2)=0$ for some denominator $u$; hence
$(ua)^2=(ub)^2$. Applying the property and then localizing gives
$a/s=b/s$ or $a/s=-b/s$. The case where one element is zero follows
from reducedness.  In a reduced ring in which $2$ is a unit, if
$uv=0$, take $a=(u+v)/2$ and $b=(u-v)/2$.  Equality of their squares implies
$u=0$ or $v=0$.  Thus a nonzero localization $A[1/2]$ is a domain.

The third assertion follows from the kernel computation.  The final
assertion follows because localization gives a bijection between the
minimal primes avoiding $2$, and a reduced nonzero ring is a domain exactly
when it has one minimal prime.
\end{proof}

\subsection{Permanence and classical factorization classes}

Passing a factorization through a homomorphism requires control of its
kernel. The following result establishes the basic transfer properties
and makes the required regularity conditions explicit.

\begin{proposition}\label{prop:permanence}
Let $R$ be a ring.
\begin{enumerate}
\item If $R$ is SAF and $S\subseteq\Reg(R)$ is multiplicative, then
$S^{-1}R$ is SAF. If $R$ is SSAF, every localization of $R$ is SSAF.
\item Let $f:R\twoheadrightarrow B$ be surjective. If $R$ is SSAF,
then $B$ is SSAF. If $R$ is SAF and each regular element of $B$ has
a regular lift in $R$, then $B$ is SAF. If $R$ is SAF and $\ker f$
is regular, then $B$ is SSAF.
\item A finite product of nonzero rings is SAF, respectively SSAF,
exactly when each of its factors has the corresponding property.
\end{enumerate}
\end{proposition}
\begin{proof}
A radical sdf ideal remains radical sdf after localization, unless it
becomes the unit ideal: apply Proposition~\ref{prop:sdf-localization}
to its quotient. If $S\subseteq\Reg(R)$, a regular element of
$S^{-1}R$ has regular numerator, so every regular ideal contracts to
a regular ideal. Factor the contraction and extend the factors.
For SSAF, first use Corollary~\ref{cor:ssaf-ssp} to choose radical
sdf factors, then apply the same extension argument.

Let $J$ be a proper ideal of $B$, and factor $f^{-1}(J)$ whenever the
hypothesis permits. Each factor contains $f^{-1}(J)$ and hence
$\ker f$. Its image is therefore proper and sdf: nonzero testing
elements of $B$ have nonzero lifts, and the difference-of-squares
membership lifts because the kernel is contained in the factor.
The inverse image is regular whenever $J$ has a regular lift, or
whenever $\ker f$ is regular. This proves (2).

For two factors, write a regular ideal of $R_1\times R_2$ as
$I_1\times I_2$, where each $I_i$ is regular. Factor each proper
$I_i$ and lift its factors to the product, filling the other coordinate
with the whole ring. These lifted factors are radical sdf, by the
quotient criterion. For SSAF use radical factorizations supplied by
Corollary~\ref{cor:ssaf-ssp}, which is necessary when zero was
originally an exceptional sdf factor. Conversely factor
$I_1\times R_2$ and apply the first projection; its kernel is contained
in every factor. Induction proves the finite-product statement.
\end{proof}

The transfer statement also gives a useful form of descent for
amalgamated algebras.

\begin{corollary}\label{cor:amalgamation-descent}
Let $f:A\to B$ be a homomorphism and $J$ an ideal of $B$.
If $A\bowtie^f J=\{(a,f(a)+j):a\in A,\ j\in J\}$ is SSAF,
then $A$ is SSAF.
\end{corollary}
\begin{proof}
The first projection is a surjection. Apply
Proposition~\ref{prop:permanence}(2).
\end{proof}

Augmentation gives coefficient descent without a restriction on the
units of the monoid. A commutative monoid $\Gamma$ is written additively
with identity $0$; its monoid algebra $A[\Gamma]$ consists of finite
sums $\sum_\gamma a_\gamma X^\gamma$, with
$X^\gamma X^\delta=X^{\gamma+\delta}$.

\begin{corollary}\label{cor:monoid-descent}
Let $A$ be a commutative ring and $\Gamma$ a commutative monoid.
If $A[\Gamma]$ is SSAF, then $A$ is SSAF. If $A[\Gamma]$ is SAF,
then $A$ is SAF. In particular, these assertions apply to $A[X]$.
\end{corollary}
\begin{proof}
The augmentation $\varepsilon:A[\Gamma]\to A$, defined by
$\varepsilon(\sum_\gamma a_\gamma X^\gamma)=\sum_\gamma a_\gamma$,
is a surjective homomorphism. A regular element of $A$ is regular
as a constant in $A[\Gamma]$, because multiplication by it is
injective on the free $A$-module with basis
$\{X^\gamma:\gamma\in\Gamma\}$. Both assertions follow from
Proposition~\ref{prop:permanence}(2).
\end{proof}

If the only unit of $\Gamma$ is $0$, the constant-coefficient map
$\pi_0:A[\Gamma]\to A$ is also a ring homomorphism: the equality
$\gamma+\delta=0$ forces $\gamma=\delta=0$. Consequently
$\pi_0(IJ)=\pi_0(I)\pi_0(J)$ for all ideals $I,J$ of $A[\Gamma]$.
The augmentation proof above applies even when $\Gamma$ has other units.

We next record the precise relation with the established radical
factorization classes. A \emph{special primary ring} is a local ring
whose proper ideals, including zero, are powers of its maximal ideal.
An \emph{almost multiplication ring} is a ring whose prime localizations
are DVRs or special primary rings. For a Marot ring we use
$R_{[P]}=\{x\in T(R):sx\in R\text{ for some }s\in R\setminus P\}$;
this equals localization at $\Reg(R)\setminus P$
\cite[Theorem~7.6]{Huckaba1988}. Such a ring is an \emph{$N$-ring}
if $R_{[M]}$ is a discrete rank-one Manis valuation ring for each
regular maximal ideal $M$ \cite{Griffin1969,Larsen1970}.
Here the Manis valuation takes values in $\mathbb Z\cup\{\infty\}$,
is surjective on its finite values, satisfies the usual multiplicative
and ultrametric rules, and realizes the valuation ring as the elements
of nonnegative value in its total quotient ring.

\begin{proposition}\label{prop:classical-comparison}
The following statements hold.
\begin{enumerate}
\item SAF implies SP, and SSAF implies SSP. If $2R=0$, each of
these implications is an equivalence.
\item If $2$ is a unit, SAF is equivalent to the Dedekind-ring
property, and SSAF is equivalent to ZPI.
\item For a Marot ring, SP implies the $N$-ring property, and
regular-Noetherian SAF-rings are exactly Dedekind rings.
\item Every SSP-ring is an almost multiplication ring.
For a Noetherian ring the properties SSAF, SSP, and ZPI are equivalent;
every ZPI-ring is Noetherian.
\end{enumerate}
\end{proposition}
\begin{proof}
Every factor of a regular ideal is nonzero, so its sdf factors are
radical. The strong implication follows from
Corollary~\ref{cor:ssaf-ssp}. When $2R=0$, radical ideals are sdf by
Proposition~\ref{prop:sdf-localization}; hence radical factorizations
are sdf factorizations. When $2$ is a unit, radical sdf ideals are
prime by that proposition. Choosing radical sdf factors proves (2),
and the reverse implications follow because prime ideals are sdf.

For (3), apply \cite[Theorem~2.6 and Corollary~2.7]{AhmedDumitrescu2017},
which are stated for Marot rings. The implication in (4) is
\cite[Theorem~3.3]{AhmedDumitrescu2017}; its Corollary~3.5 identifies
Noetherian SSP-rings with ZPI-rings. ZPI implies SSAF directly.
The classical structure theorem for ZPI-rings expresses them as finite
products of Dedekind domains and special primary rings, and in
particular makes them Noetherian \cite[Theorems 1 and 2]{Wood1969}.
\end{proof}

Figure~\ref{fig:hierarchy} summarizes these implications. The final
arrow in the bottom row uses the Marot hypothesis; the other arrows
hold under the definitions above.

\begin{figure}[ht]
\centering
\begin{tikzpicture}[every node/.style={font=\small,align=center}]
\node (z) at (0,1.6) {ZPI};
\node (ssaf) at (3.3,1.6) {SSAF};
\node (ssp) at (6.6,1.6) {SSP};
\node (am) at (10,1.6) {almost multiplication};
\node (d) at (0,0) {Dedekind ring};
\node (saf) at (3.3,0) {SAF};
\node (sp) at (6.6,0) {SP};
\node (n) at (10,0) {$N$-ring};
\draw[->] (z)--(ssaf); \draw[->] (ssaf)--(ssp);
\draw[->] (ssp)--(am); \draw[->] (d)--(saf);
\draw[->] (saf)--(sp); \draw[->] (sp)--(n);
\draw[->] (z)--(d); \draw[->] (ssaf)--(saf);
\draw[->] (ssp)--(sp);
\end{tikzpicture}
\caption{Regular-ideal and all-ideal factorization properties.}
\label{fig:hierarchy}
\end{figure}

The regularity restriction in Proposition~\ref{prop:permanence} is
essential even for Noetherian Marot rings.

\begin{example}\label{ex:quotient-counterexample}
Let $k$ be a field, $A=k[[x,y]]$, $\mathfrak m=(x,y)$,
$E=A/\mathfrak m$, and $R=A\ltimes E$, where
$(a,u)(b,v)=(ab,av+bu)$. If $a\in\mathfrak m$, then
$(a,u)(0,1)=0$; otherwise $(a,u)$ is a unit. Thus all regular
elements of $R$ are units. In particular, $R$ is Marot, SP, and SAF.
The nonzero ideal $P=0\ltimes E$ is prime and hence sdf,
but $R/P\cong A$ has dimension two and is not SP or SAF.
Since $P$ is nilpotent, $\Spec(R)$ and $\Spec(A)$ are homeomorphic;
in particular, $\dim R=2$.
Also $R$ is not SSAF, since SSAF descends to quotients; in particular,
it is not ZPI.
For failure of arbitrary localization, use
$B=k[[x,y,z]]\ltimes k$, equivalently
$B\cong k[[u,x,y,z]]/(u^2,ux,uy,uz)$.
The same argument makes $B$ SAF, whereas
localization at $(x,y)\ltimes k$ kills the second coordinate and
is the two-dimensional regular local domain
$k[[x,y,z]]_{(x,y)}$, which is not SAF.
\end{example}

The next examples show that the distinction between regular and
arbitrary ideals persists both for arbitrary domain quotients and in
a small Artinian ring.

\begin{example}\label{ex:total-quotient-and-artinian}
\begin{enumerate}
\item Let $A$ be a domain with fraction field $K$, and put
$R=A\ltimes K/A$. Every regular element of $R$ is a unit.
Indeed, if $0\ne a\in A$ is not a unit, then
$1/a+A$ is a nonzero element of $K/A$ annihilated by $a$, so
$(a,e)$ is a zero divisor for every $e\in K/A$.
Elements with zero first coordinate are zero divisors as well,
whereas a unit first coordinate gives a unit of $R$.
Thus $R$ is a total quotient ring and is SAF, while its quotient
$R/(0\ltimes K/A)$ is the arbitrary domain $A$.
This also gives a uniform obstruction to unrestricted quotient
descent; compare \cite[Theorem~3.5]{AndersonWinders2009}.
\item The ring $B=\mathbb F_2[x,y]/(x^2,xy,y^2)$ is SAF but not SSP,
and hence not SSAF. Its maximal ideal $\mathfrak m=(x,y)B$ satisfies
$\mathfrak m^2=0$, so every regular element is a unit. The only proper
radical ideal is $\mathfrak m$, but $xB$ is neither $\mathfrak m$ nor
zero and therefore is not a product of proper radical ideals.
\end{enumerate}
\end{example}

\section{Domains and the topology of the odd locus}\label{sec:domains}

Write $D_2=D[1/2]$, allowing $D_2=0$ in characteristic $2$, and put
$U_2=\{M\in\operatorname{Max}(D):2\notin M\}$.
The inverse topology on $\operatorname{Max}(D)$ has a basis consisting of
$V(J)\cap\operatorname{Max}(D)$ for finitely generated ideals $J$.
The reference to the inverse topology is essential: for example, the
Zariski topology on the maximal spectrum of $\mathbb Z$ is not discrete.

\begin{theorem}\label{thm:domain-criterion}
Let $D$ be an SP-domain.  The following conditions are equivalent:
\begin{enumerate}
\item $D$ is an SAF-domain;
\item $D_2$ is zero or a Dedekind domain, allowing fields;
\item every $0\ne a\in D$ belongs to only finitely many members of $U_2$;
\item every member of $U_2$ is finitely generated, equivalently invertible;
\item $U_2$ is discrete in the inverse topology.
\end{enumerate}
In particular, every non-finitely-generated maximal ideal of an
SAF-domain contains $2$.
\end{theorem}

\begin{proof}
Fields and domains of characteristic $2$ satisfy the assertions
directly, using Proposition~\ref{prop:sdf-localization}. Assume
that $D$ is not a field and has characteristic different from $2$.  The implication
from (1) to (2) follows by localization: in $D_2$ every nonzero sdf-absorbing
ideal is prime, so every nonzero ideal is a finite product of primes.

Assume (2).  Let $J$ be a nonzero proper radical ideal of $D$.  Since
$D/J$ is a reduced zero-dimensional ring, it is von Neumann regular.
In this quotient, the principal ideal generated by $2$ and its annihilator
are complementary idempotent ideals.  Therefore
\[
   J=(J+2D)(J:2),\qquad (J+2D)+(J:2)=D.
\]
The first factor is radical and contains $2$, hence is sdf-absorbing when
proper.  In the Dedekind domain $D_2$, the radical ideal $JD_2$ is a finite
product of distinct maximal ideals.  Contracting yields
$ (J:2)=M_1\cdots M_s$, where the $M_i$ are precisely the maximal ideals
over $J$ avoiding $2$.  To justify contraction, both sides are radical,
their maximal supports agree, and pairwise comaximal ideals have product
equal to intersection.  Thus $J$ has an sdf-factorization.  Applying this
to the radical factors of an arbitrary nonzero ideal proves (1).

Condition (2) implies (3), since a nonzero element of a Dedekind domain
belongs to finitely many maximal ideals.  Conversely, (3) says that the
almost Dedekind domain $D_2$ has finite character, which makes it Dedekind.

For (1) implying (4), choose $0\ne a\in M\in U_2$ and factor $aD$ into
sdf-absorbing ideals.  Every factor is invertible, since a factor of an
invertible ideal is invertible.  Choose a factor $J\subseteq M$.
The preceding localization criterion says that $M$ is the only maximal
ideal over $J$ avoiding $2$.  Hence $(J:2)=M$, and the displayed
comaximal product expresses $J$ as $(J+2D)M$.  Thus $M$ is invertible.

If (4) holds, $\{M\}=V(M)$ is inverse-open for every $M\in U_2$, proving
(5).  Conversely suppose (5), and fix $0\ne a\in D$.
The space $V(a)$ is compact in the inverse topology: it identifies with
the spectrum of the zero-dimensional ring $D/aD$, on which the induced
inverse and Zariski topologies agree.  It is also Hausdorff.  The set $U_2$
is inverse-clopen.  Indeed $V(2)$ is a basic inverse-open set, and for
$M\notin V(2)$ one can choose $m\in M$ with $(2,m)=D$, so $V(m)$ is an
inverse-open neighborhood disjoint from $V(2)$.  It follows that
$V(a)\cap U_2$ is compact and discrete, hence finite.  This proves (3).
\end{proof}

\begin{remark}
For any nonzero proper ideal $I$ of an SP-domain, $V(I)$ is compact,
Hausdorff, and zero-dimensional in the inverse topology, because it is
the spectrum of a zero-dimensional quotient.  In an SAF-domain its odd
part is finite and consists of isolated points.  The even part may be an
infinite Boolean space.  This gives a precise topological distinction
between SAF-domains and arbitrary SP-domains.
\end{remark}

The preceding theorem includes the characteristic-$2$ and $2$-invertible
cases, and identifies the additional possibility of mixed characteristic.

\begin{corollary}\label{cor:domain-specializations}
Let $D$ be an integral domain.
\begin{enumerate}
\item If $2\in U(D)$, then $D$ is SAF if and only if it is Dedekind.
\item If $D$ is SP and every maximal ideal of $D$ contains $2$, then
$D$ is SAF.
\item $D$ is SAF if and only if it is SP and $D[1/2]$ is Noetherian,
with the zero ring allowed in characteristic $2$.
\end{enumerate}
\end{corollary}
\begin{proof}
The first assertion is immediate. For the second, the set $U_2$ is
empty. For the third, a nonzero localization of an SP-domain remains
almost Dedekind, and a Noetherian almost Dedekind domain is Dedekind.
\end{proof}

The domain criterion also gives permanence for all intermediate
rings in the quotient field.

\begin{corollary}\label{cor:saf-overrings}
Every overring of an SAF-domain is an SAF-domain.
\end{corollary}
\begin{proof}
Let $D\subseteq T\subseteq\operatorname{Frac}(D)$, with $D$
SAF. For a prime $Q$ of $T$, put $P=Q\cap D$. The
valuation domain $D_P$ lies in $T_Q$. If
$x\in T_Q\setminus D_P$, then
$x^{-1}\in PD_P\subseteq QT_Q$, contradicting
$1=xx^{-1}\in QT_Q$. Hence $T_Q=D_P$ and
$Q=PD_P\cap T$, which also proves injectivity of contraction.
Also every ideal $J$ of $T$ equals $(J\cap D)T$.
Indeed, locally at $Q$, an element $x$ of $J$ lies in
$D_{Q\cap D}$, so some $c\in D\setminus(Q\cap D)$ satisfies
$cx\in D\cap J$. This verifies the equality at every maximal
localization of $T$.

For $J\ne0$, its contraction is nonzero. Extend a radical
factorization of $J\cap D$ to $T$. Each extended factor is
radical, as can be checked in $T_Q=D_{Q\cap D}$; omit factors
that extend to $T$. Thus $T$ is
SP. If $0\ne x=a/b\in T$ lies in a maximal ideal $Q$ avoiding
$2$, then $a\in Q\cap D$, and $Q\cap D$ is a nonzero maximal
ideal of $D$ avoiding $2$. There are only finitely many such
contractions containing $a$, by Theorem~\ref{thm:domain-criterion},
and contraction is injective. Hence $T$ has finite character on
its maximal ideals avoiding $2$. The same theorem proves that
$T$ is SAF.
\end{proof}

\begin{remark}\label{rem:zero-divisor-localization}
The domain hypothesis cannot be dropped. Let
$A=\Z[x,y]_{(2,x,y)}$, let $E=A/(2,x,y)A\cong\F_2$, and put
$R=A\ltimes E$. Here $(a,u)(b,v)=(ab,av+bu)$ defines the
idealization. An element $(a,u)$ is regular exactly when $a$ is a
unit of $A$: if $a$ belongs to the maximal ideal, $(a,u)(0,1)=0$;
if $a$ is a unit, $(a,u)$ is a unit. Thus $R$ has no regular proper
ideals and is a Marot SAF-ring. However, $R[1/2]\cong A[1/2]$,
and the chain $(0)\subsetneq(x)\subsetneq(x,y)$ survives in this
localization. It is a domain of dimension at least two and is not
Dedekind.
\end{remark}

\subsection{An ISP-domain comparison}

An ISP-domain is a domain in which every nonzero proper ideal is an
invertible ideal times a nonempty finite product of proper radical ideals.
For an almost Dedekind domain, weak factorization means that each
nonzero nondivisorial ideal is the product of its divisorial
closure and finitely many maximal ideals.  The following is a
direct consequence of the ISP--SP equivalence in
\cite[Theorem~13]{AhmedDumitrescu2018ISP} and the marked domain
criterion.

\begin{corollary}\label{cor:isp-weak-saf}
Let $D$ be a one-dimensional ISP-domain with weak factorization.
Then $D$ is SAF if and only if every maximal ideal not containing
$2$ is invertible.  Equivalently, $D[1/2]$ is zero or Dedekind,
with fields included in the latter convention.  In particular,
if $D$ has only finitely many noninvertible maximal ideals and all
of them contain $2$, then $D$ is SAF.
\end{corollary}
\begin{proof}
By \cite[Corollary~4]{AhmedDumitrescu2018ISP}, $D$ is almost
Dedekind.  The weak factorization hypothesis and
\cite[Theorem~13]{AhmedDumitrescu2018ISP} imply that $D$ is SP,
so Theorem~\ref{thm:domain-criterion} proves the equivalences.
An almost Dedekind domain with only finitely many noninvertible
maximal ideals has weak factorization, as recalled immediately
before \cite[Theorem~13]{AhmedDumitrescu2018ISP}; alternatively,
\cite[Corollary~14]{AhmedDumitrescu2018ISP} directly gives the
ISP--SP implication under this finiteness assumption.  The same
domain criterion then proves the last assertion.
\end{proof}

\subsection{Strictness examples}

The following construction gives both characteristic-$2$ and
$2$-invertible examples without an unproved transfer assumption.
It is the familiar successive-splitting construction underlying
\cite[Example~3.4.1]{FontanaHoustonLucas2013} and
\cite[p.~817]{AhmedDumitrescu2018ISP}.

\begin{example}\label{ex:successive-splitting}
Let $k$ be a field and let $y,x_1,x_2,\ldots$ be algebraically
independent over $k$. Put $z_n=y/(x_1\cdots x_n)$ and define
\[
 C_n=k[x_1,\ldots,x_n,z_n]_{(x_1)\cup\cdots\cup(x_n)\cup(z_n)},
 \qquad C=\bigcup_{n\geq1}C_n.
\]
The subscript denotes localization at the complement of the displayed
union of prime ideals. Each $C_n$ is a semilocal PID: the polynomial
ring is a UFD, and precisely the indicated height-one primes survive.
The equality $z_n=x_{n+1}z_{n+1}$ yields inclusions
$C_n\subseteq C_{n+1}$, because each new prime contracts to $(z_n)$
and each old prime $(x_i)$ contracts to itself. Thus $C$ is B\'ezout.

Every nonzero element of $C_n$ is a unit times a product of powers of
$x_1,\ldots,x_n,z_n$, and hence divides a power of $y$ in $C_n$.
Every maximal ideal of $C$ consequently contains $y$. Its contractions
either eventually select a fixed $(x_i)$ or always select $(z_n)$.
It follows that the maximal ideals are
$M_i=x_iC$ for $i\geq1$ and $N=\bigcup_n z_nC$.
Their quotients are directed unions of residue fields, proving
maximality. At $M_i$ the local directed system has the common
uniformizer $x_i$, while at $N$ it has common uniformizer $y$.
A union of these DVRs is again a DVR, since all nonzero valuations
are integers and every nonzero ideal has a least value. Moreover
$yC_M=MC_M$ for every maximal ideal $M$.

Thus $C$ is almost Dedekind and $yC$ is a nonzero principal radical
ideal contained in every maximal ideal. The standard SP criterion
makes $C$ an SP-domain. The inclusions
$z_nC\subsetneq z_{n+1}C$ are strict after localizing at $M_{n+1}$,
so $N$ is not finitely generated and $C$ is not Dedekind.
If $k=\mathbb F_2$, then $C$ is SAF. If $k=\mathbb Q$, then $2$ is a
unit and $C$ is not SAF. These two choices show strictness of
Dedekind $\Rightarrow$ SAF and SAF $\Rightarrow$ SP, and also of
SSAF $\Rightarrow$ SSP in the domain case.
\end{example}

An algebraic tower of Laurent polynomial rings gives a second explicit
non-Noetherian SP-domain in which $2$ is invertible.

\begin{example}\label{ex:real-dyadic}
Let $z_n=X^{1/2^n}$ for $n\geq0$, with $z_n=z_{n+1}^2$, and put
\[
 D_n=\mathbb R[z_n,z_n^{-1}],\qquad
 D=\bigcup_{n\geq0}D_n.
\]
Then $D$ is a B\'ezout SP-domain which is neither Dedekind nor SAF.
Each $D_n$ is a Laurent polynomial PID. Every finite family of elements
of $D$ lies in one stage, so every finitely generated ideal of $D$ is
principal. The extensions $D_n\subseteq D_{n+1}$ are integral.

Fix a maximal ideal $M$ of $D$. Its contraction to $D_0$ is generated
by an irreducible polynomial $p(z_0)$ with $p(0)\ne0$.
Since the characteristic is zero, $p(T^{2^n})$ is squarefree: its
derivative is $2^nT^{2^n-1}p'(T^{2^n})$, and the two polynomials
have no common root. Hence $pD_n$ is radical at every stage.
It follows that $pD$ is radical as well, since any membership
$f^r\in pD$ is witnessed in a sufficiently large finite stage.

Set $M_n=M\cap D_n$. The local ring $(D_n)_{M_n}$ is a DVR in
which the same element $p$ is a uniformizer, because $pD_n$ is
radical. Its units remain units in every later local ring.
Consequently $D_M=\bigcup_n(D_n)_{M_n}$ is a DVR: each nonzero
element has an integer valuation, and any nonzero ideal has an
element of least valuation. Thus $D$ is almost Dedekind, and every
maximal ideal contains a nonzero principal radical ideal $pD$.
The SP criterion in the preliminaries proves that $D$ is SP.

Finally, $z_n-1=(z_{n+1}-1)(z_{n+1}+1)$ gives a strict chain
$(z_n-1)D\subsetneq(z_{n+1}-1)D$. To see strictness, the least
and greatest exponents add under multiplication of nonzero finite
Laurent sums, so a unit of $D$ is a nonzero monomial. The element
$z_{n+1}+1$ is therefore not a unit. Hence $D$ is not Noetherian
and not Dedekind. Since $2$ is a unit, Theorem~\ref{thm:domain-criterion}
shows that $D$ is not SAF.
\end{example}

The remaining strictness, from SP to almost Dedekind in the domain
case, is furnished by the classical almost Dedekind domains without
radical factorization in \cite{VaughanYeagy1978}; their obstruction is
studied systematically in \cite{Spirito2023}. These are also
almost multiplication rings which are not SSP. The statement uses
those established examples and does not assert that the preceding
successive-splitting domain has that property.

\section{Torsion splitting and the prime spectrum}\label{sec:torsion}

Localization at a single element separates the residue-characteristic
components of a torsion module. We first give explicit formulas over
SP-domains, then identify the exact prime-theoretic condition over an
arbitrary domain. The latter condition shows which parts of the
splitting depend only on dimension and which use radical factorization.

For a $D$-module $M$ put
$t_2(M)=\{m\in M:2^nm=0\text{ for some }n\ge0\}$.
This is the torsion radical associated to the multiplicative set
$\{1,2,2^2,\ldots\}$; localization has kernel $t_2(M)$.

\begin{theorem}\label{thm:cyclic-torsion}
Let $D$ be an SP-domain, $0\ne I\subsetneq D$, and let
$n\ge\sup_{M\in\operatorname{Max}(D)}\nu_I(M)$, where $\nu_I(M)$ denotes the
exponent of $ID_M$ in the DVR $D_M$.  Define
$I_{\rm ev}=I+2^nD$ and $I_{\rm odd}=(I:2^n)$.  Then
\[
 I=I_{\rm ev}I_{\rm odd},\qquad I_{\rm ev}+I_{\rm odd}=D,
 \qquad I_{\rm odd}=(I:2^\infty).
\]
There is a canonical decomposition
\[
 D/I\cong D/I_{\rm ev}\times D/I_{\rm odd},\qquad
 t_2(D/I)\cong D/I_{\rm ev},\qquad
 D/I_{\rm odd}\cong D_2/ID_2.
\]
In particular, the localization sequence of the cyclic module $D/I$
splits uniquely, and $2^n t_2(D/I)=0$.
\end{theorem}

\begin{proof}
The bound is finite because any finite radical factorization of $I$
bounds all its local exponents. If $2=0$, the formulas hold
with $I_{\rm ev}=I$ and $I_{\rm odd}=D$. Assume $2\ne0$.
At a maximal ideal $M$ containing $2$,
the inequality $v_M(2^n)\ge n\ge \nu_I(M)$ yields
$(I_{\rm ev})_M=I_M$ and $(I_{\rm odd})_M=D_M$.  At a maximal ideal not
containing $2$, the roles reverse.  Local equality of ideals proves both
the product formula and comaximality.  It also identifies $I_{\rm odd}$
with the saturation.  (Localization of the colon by a single element
commutes with localization.)

The Chinese remainder theorem gives the asserted ring product.  The even
quotient is killed by $2^n$, whereas $2$ is a unit in the odd quotient:
every maximal ideal of the latter avoids $2$.  Thus these are exactly the
torsion and localized components.  A homomorphism from a module on which
$2$ is invertible to a module killed by $2^n$ is zero; hence the section of
localization is unique.
\end{proof}

\begin{corollary}
Let $D$ be an SAF-domain and $0\ne I\subsetneq D$.  Write
$\mathcal O(I)=\{M\supseteq I:2\notin M\}$ and
$\nu_M=\nu_I(M)$.  Then $\mathcal O(I)$ is finite and
\[
 (D/I)[1/2]\cong\prod_{M\in\mathcal O(I)}D/M^{\nu_M},\qquad
 \operatorname{length}_{D_2}((D/I)[1/2])
       =\sum_{M\in\mathcal O(I)}\nu_M.
\]
The empty product is the zero module/ring component.
\end{corollary}

\begin{proof}
Factor $ID_2$ in the Dedekind domain $D_2$ and use the Chinese remainder
theorem.  Each odd maximal ideal of $D$ is invertible by the domain
criterion.  Its powers contract correctly, since their localizations at
the maximal ideals are the expected DVR powers.  Each quotient
$D/M^{\nu_M}$ has length $\nu_M$.
\end{proof}

\begin{corollary}\label{cor:module-criterion}
For an SP-domain $D$, the SAF property is equivalent to each of the
following module conditions:
\begin{enumerate}
\item $(D/I)[1/2]$ has finite length over $D_2$ for every nonzero ideal $I$;
\item $M[1/2]$ has finite length over $D_2$ for every finitely generated
torsion $D$-module $M$.
\end{enumerate}
\end{corollary}

\begin{proof}
The first condition forces every nonzero principal ideal to have finite
odd support, and the domain criterion applies.  In an SAF-domain, a
finitely generated torsion module is annihilated by some $0\ne a\in D$;
its localization is a finitely generated module over the Artinian ring
$D_2/aD_2$, so has finite length.  The cyclic case gives the converse.
\end{proof}

\subsection{A splitting criterion over arbitrary domains}

The existence of a torsion splitting uses less than radical
factorization.  Idempotent localizations and their split torsion
theories are characterized in
\cite[Theorems~2.6 and 4.1]{KimKoc2026Strong}.
The next result applies that mechanism simultaneously to all
nonzero cyclic quotients of a domain and identifies the exact
obstruction in its prime spectrum.  We use an arbitrary nonzero
element $c$ so that the role of $2$ can be separated from the
general localization argument.

\begin{theorem}\label{thm:torsion-prime-chain-criterion}
Let $D$ be an integral domain with quotient field $K$, and let
$c\in D$ be nonzero and not a unit.  Put
$t_c(M)=\{x\in M:c^n x=0\text{ for some }n\geq0\}$.
The following conditions are equivalent.
\begin{enumerate}
\item For every torsion $D$-module $M$, the localization sequence
\[
 0\longrightarrow t_c(M)\longrightarrow M
   \longrightarrow M[1/c]\longrightarrow0
\]
is exact and has a unique $D$-linear section; these sections are
functorial in $M$.
\item The map $D/aD\longrightarrow(D/aD)[1/c]$ is surjective
for every $0\ne a\in D$.
\item $D_{1+cD}[1/c]=K$.
\item There are no prime ideals $P\subseteq Q$ such that
$P\ne0$, $c\notin P$, and $c\in Q$.
\end{enumerate}
\end{theorem}
\begin{proof}
The first condition immediately implies the second.  Suppose (2)
holds, and lift $1/c$ in $(D/aD)[1/c]$ to some $r\in D/aD$,
using the same notation for a lift to $D$.  Equality in the
localization gives $c^n(1-cr)\in aD$ for some $n\geq0$.
Thus every nonzero $a$ divides an element of the multiplicative set
\[
 T=\{c^n(1+cd):n\geq0,\ d\in D\}.
\]
Every nonzero element of $D$ is therefore invertible in $D_T$,
so $D_T=K$.  This proves (3).  Conversely, (3) supplies, for each
$a\ne0$, a relation $c^n(1-cr)\in aD$ by clearing denominators.
It follows that $1/c=r$ in $(D/aD)[1/c]$, which proves (2).

A prime ideal $P$ avoids $T$ precisely when $c\notin P$ and
$P+cD\ne D$.  The latter condition is equivalent to the existence
of a maximal ideal $Q$ containing $P$ and $c$.  Consequently
$D_T$ has no nonzero prime ideal exactly when (4) holds.  Since
$D_T$ is a domain with quotient field $K$, this is equivalent to
$D_T=K$, proving the equivalence of (3) and (4).

For completeness, we give the idempotent argument that proves
(3) implies (1), including uniqueness for modules with no common
nonzero annihilator.  Fix $0\ne a\in D$ and set $R=D/aD$.
The relation obtained above has the form
$c^n=c^{n+1}b$ in $R$, where we may increase $n$ to ensure $n\geq1$.
The element $e=(cb)^n$ then satisfies
\[
 e^2=e,\qquad c^n(1-e)=0,\qquad c(be)=e.
\]
Indeed, multiplying $c^n=c^{n+1}b$ by $b^n$ shows that
$(cb)^n=(cb)^{n+1}$, and iteration proves the first identity;
iteration of the original relation proves the second.
Thus $c$ is invertible on $eR$, while it is nilpotent on $(1-e)R$.
For every module $N$ annihilated by $a$ this gives
\[
 N=(1-e)N\oplus eN,\qquad
 t_c(N)=(1-e)N,\qquad N[1/c]\simeq eN.
\]
The projection is canonical: a homomorphism from a module on
which $c$ is invertible to a module killed by $c^n$ is zero.
Any two choices of nonzero annihilators can be compared using
their product.  Hence these splittings commute with homomorphisms
between modules having a common nonzero annihilator.

Every finitely generated submodule of a torsion module over a
domain has a common nonzero annihilator.  The compatible splittings
therefore pass to their directed union and give (1), since
localization commutes with directed limits.  To check global
uniqueness, let $L$ be a torsion module on which $c$ is invertible,
let $T'=t_c(T')$, and let $f:L\to T'$ be $D$-linear.
For $x\in L$, choose $a\ne0$ with $ax=0$ and use the idempotent
$e$ of $D/aD$ above.  Then $ex=x$, because $(1-e)x$ is killed
by a power of $c$.  The element $ef(x)$ lies in the $c$-invertible
summand of $T'[a]=\{y\in T':ay=0\}$ and is also killed by a power
of $c$, so $ef(x)=0$.  Hence $f(x)=f(ex)=ef(x)=0$.
Thus $\operatorname{Hom}_D(L,T')=0$, proving uniqueness of the
section and its functoriality.
\end{proof}

The dimension-one case is a consequence of this prime-chain
criterion.  Equivalently, one can apply the zero-dimensional
idempotent-localization criterion of
\cite[Theorems~3.4 and 4.1]{KimKoc2026Strong} to $D/aD$.

\begin{example}\label{ex:torsion-dimension-two}
For $D=\Z[X]$ and $c=2$, the chain
$(0)\subsetneq(X)\subsetneq(2,X)$ violates condition (4).
Correspondingly, the torsion module $D/XD\simeq\Z$ maps to
$\Z[1/2]$ by a localization map that is not surjective.
\end{example}

Recall that a module over a domain is torsion if each of its
elements is annihilated by a nonzero scalar.  The splitting below
holds for all domains of dimension at most one.

\begin{theorem}\label{thm:all-torsion}
Let $D$ be an integral domain of dimension at most one and let
$M$ be a torsion $D$-module.  Then localization induces a
functorial, uniquely split exact sequence
\[
 0\longrightarrow t_2(M)\longrightarrow M
   \longrightarrow M[1/2]\longrightarrow0.
\]
\end{theorem}
\begin{proof}
If $2=0$ or $2$ is a unit, the assertion is immediate.
Otherwise every nonzero prime ideal is maximal, so the prime
chain excluded by Theorem~\ref{thm:torsion-prime-chain-criterion}
cannot occur.  Apply that theorem with $c=2$.
\end{proof}

\section{Complementary overrings and ideal classes}\label{sec:complementary-overrings}

The torsion decomposition has a counterpart for overrings and invertible
ideals.  The two localizations below retain, respectively, the maximal
ideals containing $2$ and those avoiding $2$.  Their invertible ideal
groups separate, while principal ideals contribute an explicit
obstruction to separating the Picard group.

For a domain $A$ with quotient field $K$, write $\operatorname{Inv}(A)$
for its group of nonzero invertible fractional ideals and
$\operatorname{Pic}(A)=\operatorname{Inv}(A)/\{aA:a\in K^\times\}$.
We order invertible ideals by reverse inclusion.  For a field both
groups are trivial.

The patching statement holds for every one-dimensional Pr\"ufer
domain.  Radical factorization enters only in the final SAF criterion.

\begin{theorem}\label{thm:complementary-overrings}
Let $D$ be a one-dimensional Pr\"ufer domain, let $K$ be its
quotient field, and suppose
that $2$ is nonzero and not a unit.  Put
\[
 E=V(2D),\qquad O=\Max(D)\setminus E,\qquad
 T_{\rm e}=D_{1+2D},\qquad T_{\rm o}=D[1/2].
\]
Then the following assertions hold.
\begin{enumerate}
\item The nonzero maximal ideals of $T_{\rm e}$ and $T_{\rm o}$
correspond by contraction to $E$ and $O$, respectively, and their
local rings are the corresponding $D_M$.  Moreover,
\[
 D=T_{\rm e}\cap T_{\rm o},\qquad
 T_{\rm e}+T_{\rm o}=K.
\]
In particular, $0\longrightarrow D\longrightarrow
T_{\rm e}\oplus T_{\rm o}\longrightarrow K\longrightarrow0$ is exact,
where the maps are $d\mapsto(d,d)$ and $(x,y)\mapsto x-y$.
\item Extension is an isomorphism of lattice-ordered groups
\[
 \operatorname{Inv}(D)\longrightarrow
 \operatorname{Inv}(T_{\rm e})\oplus\operatorname{Inv}(T_{\rm o}),
 \qquad I\longmapsto(IT_{\rm e},IT_{\rm o}).
\]
Its inverse sends $(J_{\rm e},J_{\rm o})$ to
$J_{\rm e}\cap J_{\rm o}$, with the intersection taken in $K$.
\item There is an exact sequence of abelian groups
\[
 0\longrightarrow
 \frac{K^\times}{T_{\rm e}^\times T_{\rm o}^\times}
 \xrightarrow{\ \delta\ }\operatorname{Pic}(D)
 \longrightarrow
 \operatorname{Pic}(T_{\rm e})\oplus\operatorname{Pic}(T_{\rm o})
 \longrightarrow0,
\]
where $\delta([a])=[aT_{\rm e}\cap T_{\rm o}]$.
\item If $D$ is SP, then $D$ is SAF if and only if $T_{\rm o}$ is Dedekind.
\end{enumerate}
\end{theorem}
\begin{proof}
The set $1+2D$ is multiplicatively closed and does not contain zero.
It avoids every maximal ideal containing $2$.  If $M\in O$, the
equality $M+2D=D$ supplies an element of $M\cap(1+2D)$.
The localization correspondence therefore gives exactly the asserted
maximal ideals; if $O$ is empty, $T_{\rm o}=K$.
The identity $D=\bigcap_{M\in\Max(D)}D_M$ now yields
$D=T_{\rm e}\cap T_{\rm o}$.

We first record the complementary ideal decomposition without
using discrete valuations.  For $0\ne F\subsetneq D$, the
proof of Theorem~\ref{thm:torsion-prime-chain-criterion}, applied
to a nonzero principal ideal inside $F$, supplies $n\geq1$ and
$b$ such that $2^n=2^{n+1}b$ in $D/F$.  With $e=(2b)^n$
in this quotient, one has
$e^2=e$, $2^n(D/F)=e(D/F)$, and
$\operatorname{Ann}_{D/F}(2^n)=(1-e)(D/F)$.
Consequently
\[
 F_{\rm e}=F+2^nD,\qquad F_{\rm o}=(F:2^n)
\]
are comaximal and have product $F$.  They localize to $F D_M$
and $D_M$, respectively, at $M\in E$, and in the reverse order
at $M\in O$.  These local descriptions show that the factors
are independent of sufficiently large $n$.  If $F$ is finitely
generated, so are both factors: besides the formula for
$F_{\rm e}$, one has $F_{\rm o}=F+(1-\widetilde e)D$ for a
lift $\widetilde e$ of $e$.  Thus both factors are invertible
when $F$ is a nonzero finitely generated ideal.  The unit ideal
has both projections equal to $D$.

For $0\ne d\in D$, write $dD=I_{\rm e}I_{\rm o}$ and choose
$x\in I_{\rm e}$, $y\in I_{\rm o}$ with $x+y=1$.
At every $M\in E$, $I_{\rm e}D_M=dD_M$, so $x/d\in T_{\rm e}$.
Similarly $y/d\in T_{\rm o}$.  Hence
$1/d=x/d+y/d$, proving $T_{\rm e}+T_{\rm o}=K$ and (1).

For an invertible fractional ideal $F$, choose $0\ne d\in D$
with $dF\subseteq D$ and set
\[
 F_{\rm e}=(dF)_{\rm e}(dD)_{\rm e}^{-1},\qquad
 F_{\rm o}=(dF)_{\rm o}(dD)_{\rm o}^{-1}.
\]
The local descriptions prove independence of $d$ and give
$F=F_{\rm e}F_{\rm o}$.  Equality, products, and the lattice
operations are detected at all maximal localizations.

Given invertible fractional ideals $J_{\rm e}$ and $J_{\rm o}$ of
the two overrings, choose finite generating sets in $K$ and let
$F,G$ be the fractional $D$-ideals generated by these sets.
They are invertible because $D$ is Pr\"ufer.  Then
$I=F_{\rm e}G_{\rm o}$ extends to the prescribed pair.
Extension is injective by the local description.  Finally,
$I=IT_{\rm e}\cap IT_{\rm o}$: if $z$ belongs to the intersection,
then $zI^{-1}\subseteq T_{\rm e}\cap T_{\rm o}=D$, so $z\in I$.
This proves (2), including the inverse formula.

The image of $aT_{\rm e}\cap T_{\rm o}$ under (2) is
$(aT_{\rm e},T_{\rm o})$, so its class is zero precisely when
there is $c\in K^\times$ with
$cT_{\rm e}=aT_{\rm e}$ and $cT_{\rm o}=T_{\rm o}$.
Equivalently, $a\in T_{\rm e}^\times T_{\rm o}^\times$.
This proves that $\delta$ is well defined and injective.
If an ideal class becomes principal over both overrings, its image
under (2) has the form $(aT_{\rm e},bT_{\rm o})$.
After multiplication by the principal ideal $b^{-1}D$, this pair
is $(ab^{-1}T_{\rm e},T_{\rm o})$, which lies in the image of
$\delta$.  Surjectivity on the right follows from (2).
Under the additional SP hypothesis, the last assertion is
Theorem~\ref{thm:domain-criterion}.
\end{proof}

For B\'ezout domains the ideal splitting can be realized by elements.
This specializes the general B\'ezout SP-domain framework of
McGovern \cite{McGovern2007} to the element $2$.

\begin{corollary}\label{cor:bezout-even-odd-elements}
Under the hypotheses of Theorem~\ref{thm:complementary-overrings},
suppose that $D$ is B\'ezout.  Every $0\ne a\in D$ has a
factorization $a=a_{\rm e}a_{\rm o}$ such that
\[
 a_{\rm e}D+a_{\rm o}D=D,\qquad
 V(a_{\rm e}D)\subseteq E,\qquad V(a_{\rm o}D)\subseteq O.
\]
The pair of principal ideals is unique.  If $D$ is also SAF,
$T_{\rm o}$ is a principal ideal domain.  In addition,
$K^\times=T_{\rm e}^\times T_{\rm o}^\times$.
\end{corollary}
\begin{proof}
The two factors of $aD$ constructed in the preceding proof
are invertible and hence principal.  Choose generators and absorb
the resulting unit into one of them.  Their local descriptions
prove uniqueness of the principal ideals and comaximality.
A localization of a B\'ezout domain is B\'ezout, so a Dedekind
localization is a principal ideal domain.  The assertion about
units follows from (3), since $\operatorname{Pic}(D)=0$.
\end{proof}

\subsection{A sharper uniform bound for irreducible divisors}
\label{subsec:picard-exponent-refinement}

Reinhart \cite[Proposition~6.8(3)]{Reinhart2014Monadic} proves
that, for an $r$-SP-monoid which is $r$-Pr\"ufer and whose
$r$-ideal class group has finite exponent $m$, every atom dividing
a power of an element $a$ already divides $a^{2m}$.  Taking the
ordinary ideal system gives this assertion for SP-domains with
bounded Picard group.  In this specialization, the last member
of the ordered radical factorization improves the uniform bound
from $2m$ to $m$.  No atomicity assumption on the whole domain
is required.

We first recall the ordered radical layers, including a direct proof
of the standard construction \cite[Lemma~2.10]{AhmedDumitrescu2017}.

\begin{lemma}\label{lem:gf-layers}
Let $D$ be an SP-domain and $I$ a nonzero proper ideal of $D$.
Then $h(I):=\max_M\nu_I(M)$ is finite, and there are nonzero radical
ideals $J_1\subseteq\cdots\subseteq J_{h(I)}$ such that
\[
 I=J_1\cdots J_{h(I)},\qquad
 V(J_j)=\{M\in\operatorname{Max}(D):\nu_I(M)\geq j\}.
\]
These ideals are uniquely determined by $I$.
\end{lemma}
\begin{proof}
Choose a radical factorization $I=Q_1\cdots Q_t$.  At every maximal
ideal $M$, each $Q_iD_M$ is either $D_M$ or $MD_M$.  Consequently
$0\leq\nu_I(M)\leq t$.  For $1\leq j\leq t$, set
\[
 J_j=\bigcap_{\substack{S\subseteq\{1,\ldots,t\}\\|S|=j}}
                         \left(\sum_{i\in S}Q_i\right).
\]
Localization commutes with these finite sums and intersections.  It
follows that $J_jD_M=MD_M$ exactly when at least $j$ of the ideals
$Q_i$ are contained in $M$; otherwise $J_jD_M=D_M$.
Thus $J_j$ is radical and has the asserted support.  For $j>h(I)$,
$J_j=D$.  The product of the other $J_j$ has the same localization as
$I$ at every maximal ideal and therefore equals $I$.  Radical ideals
are determined by their prime overideals, which proves uniqueness.
\end{proof}

The final radical layer gives the sharper bound.

\begin{proposition}\label{prop:picard-exponent-irreducibles}
Let $D$ be a nonfield SP-domain, and suppose that $\Pic(D)$ has
finite exponent $m$.  If $u\in D$ is irreducible, then
\[
 \max_{M\in\Max(D)}\nu_{uD}(M)\leq m.
\]
Moreover, for every nonzero $a\in D$,
\[
 u\mid a^n\text{ for some }n\geq1
 \quad\Longrightarrow\quad u\mid a^m.
\]
Equality in the height bound holds precisely when
$uD=P^m$ for an invertible maximal ideal $P$ whose class has
order $m$.  Conversely, a generator of $P^m$ for such an ideal
is irreducible and attains the height bound.  Choosing $a\in P\setminus P^2$ also shows that
the exponent $m$ in the second conclusion cannot be decreased
for this domain.
\end{proposition}
\begin{proof}
Write the ordered radical factorization
$uD=Q_1\cdots Q_h$ from Lemma~\ref{lem:gf-layers}, where
$Q_1\subseteq\cdots\subseteq Q_h$ are proper radical ideals.
Each $Q_i$ is invertible because their product is invertible.
For example, the fractional ideal
$(uD)^{-1}\prod_{j\ne i}Q_j$ is an inverse of $Q_i$.
Since $\Pic(D)$ has exponent $m$, there is a nonunit $b\in D$
such that $Q_h^m=bD$.  If $h\geq m$, then
\[
 uD=Q_1\cdots Q_h\subseteq Q_h^h\subseteq Q_h^m=bD.
\]
Thus $b$ divides the irreducible element $u$.  Since $b$ is a
nonunit, $uD=bD=Q_h^m$.  Choose a maximal ideal $M$ containing
$Q_h$.  It contains every $Q_i$, so comparison at the DVR $D_M$
gives $h=\nu_{uD}(M)=m$.  Therefore $h\leq m$ in all cases.
The equality $h=\max_M\nu_{uD}(M)$ proves the first assertion.

Suppose that $u\mid a^n$.  Every maximal ideal containing $u$
then contains $a$.  At such a maximal ideal,
$\nu_{uD}(M)\leq m\leq m\nu_{aD}(M)$; at every other maximal
ideal the first exponent is zero.  Hence
$a^mD_M\subseteq uD_M$ for every maximal ideal $M$.
It follows that $a^mD\subseteq uD$, which is the required
divisibility.

If $h=m$, the first part of the proof gives $uD=Q_h^m$.
If $V(Q_h)$ had two distinct points, an idempotent of the
von Neumann regular ring $D/Q_h$ would split $Q_h=AB$ into
two proper comaximal radical ideals.  They are invertible,
because their product is invertible, and $A^m$ and $B^m$
are principal.  This would factor $u$ into two nonunits.
Thus $V(Q_h)$ is a singleton, so $Q_h=P$ is an invertible
maximal ideal.  If its class had order $d<m$, the principal
ideal $P^d$ would give a proper nonunit divisor of $u$.
Hence the class has order $m$.

Conversely, let $P$ be an invertible maximal ideal whose
class has order $m$, and write $P^m=uD$.  If $u=bc$ with
nonunits $b,c$, comparison at all maximal localizations forces
$bD=P^j$ and $cD=P^{m-j}$ for an integer $0<j<m$.  This makes
$P^j$ principal, contrary to the order of the class of $P$.
Thus $u$ is irreducible, and its only nonzero local exponent
is $m$ at $P$.  Finally, $a\in P\setminus P^2$ has
$\nu_{aD}(P)=1$.  Consequently $u\mid a^m$, whereas
$u\nmid a^j$ for every $0\leq j<m$.
\end{proof}

\section{Idealizations and strong factorization}\label{sec:idealizations}

Idealizations distinguish regular-ideal factorization from factorization
of all ideals. We first characterize the regular property for arbitrary
modules. For von Neumann regular bases, we then classify the strong
property and compare it with absorbing factorization. No finiteness
assumption on the module is imposed unless explicitly stated.

\subsection{Regular ideals over an arbitrary base}

For a ring $A$ and an $A$-module $E$, the idealization $A\ltimes E$
has addition componentwise and multiplication
$(a,u)(b,v)=(ab,av+bu)$. An ideal is \emph{homogeneous} if it is
$I\ltimes F$ with $I$ an ideal, $F$ a submodule, and $IE\subseteq F$.
Products of homogeneous ideals are homogeneous, and
\[
 (I\ltimes F)(J\ltimes G)=IJ\ltimes(IG+JF).
\]
Every nonzero sdf ideal of $A\ltimes E$ is radical and therefore
has the form $J\ltimes E$; it is sdf precisely when $A/J$ is reduced
and $(A/J)[1/2]$ is zero or a domain. The zero ideal needs separate
treatment: for example it is sdf in
$\mathbb F_3\ltimes\mathbb F_3$, but does not contain
$0\ltimes\mathbb F_3$. The following is the sdf analogue of
\cite[Proposition~2.4]{AhmedDumitrescu2017}, with a direct proof
covering regular ideals that are not initially homogeneous.

\begin{proposition}\label{prop:regular-idealization}
Let $A$ be a ring, let $E$ be an $A$-module, and set
$S=\{s\in A:s\text{ acts injectively on both }A\text{ and }E\}$.
Then $A\ltimes E$ is SAF if and only if:
\begin{enumerate}
\item every proper ideal $I$ of $A$ with $I\cap S\ne\varnothing$
is a finite product of sdf-absorbing ideals of $A$;
\item $sE=E$ for every $s\in S$.
\end{enumerate}
\end{proposition}
\begin{proof}
An element $(s,u)$ of the idealization is regular exactly when
$s\in S$. Indeed, a nonzero kernel of multiplication by $s$ on $E$
gives a zero divisor $(0,v)$. If that multiplication is injective
but $sb=0$ for $b\ne0$, then $s(bu)=b(su)=0$ implies $bu=0$,
so $(s,u)(b,0)=0$. The converse follows directly from multiplication.

Suppose the idealization is SAF. Each regular ideal is a product
of nonzero radical ideals. Every such factor contains $0\ltimes E$
and is homogeneous, and products of homogeneous ideals remain
homogeneous. For $s\in S$ and $u\in E$, the principal ideal
generated by $(s,u)$ is consequently homogeneous. Thus it contains
$(0,u)$. An equality
$(0,u)=(s,u)(b,v)=(sb,sv+bu)$ forces $b=0$ and $u=sv$,
so $sE=E$. If $(s,u)$ is a unit, the same conclusion holds because
then $s$ is a unit of $A$.

For $I\cap S\ne\varnothing$, the ideal $I\ltimes E$ is regular.
In a factorization its factors have the form $J_i\ltimes E$.
Each $J_i$ contains $I$, so is nonzero; the quotient criterion
or restriction to elements $(a,0)$ shows that $J_i$ is sdf in $A$.
Taking first coordinates gives $I=\prod_iJ_i$.

Conversely, let $H$ be a regular proper ideal and choose
$(s,u)\in H$ with $s\in S$. Since $sE=E$, multiplying by
$0\ltimes E$ shows that $0\ltimes E\subseteq H$. Hence
$H=I\ltimes E$ for a proper ideal $I$ meeting $S$.
Write $I=J_1\cdots J_n$ as in (1). Each $J_i$ is nonzero and sdf,
so $J_i\ltimes E$ is sdf by its reduced quotient. Every product
of a subfamily of the $J_i$ acts surjectively on $E$, because it
contains a power of $s$. The idealization multiplication formula
therefore gives
$I\ltimes E=(J_1\ltimes E)\cdots(J_n\ltimes E)$.
\end{proof}

The two conditions involve only ideals meeting $S$; they do not imply
that $A$ is SAF. Example~\ref{ex:quotient-counterexample} illustrates
this distinction.

\begin{example}\label{ex:universal-idealization}
Let $A$ be a ring and let $\Gamma$ be the set of its regular proper
ideals having no sdf factorization. For each $I\in\Gamma$ choose a
maximal ideal $N_I$ containing $I$, and set
$E=\bigoplus_{I\in\Gamma}A/N_I$.
Then $A\ltimes E$ is SAF. Indeed, the set $S$ in
Proposition~\ref{prop:regular-idealization} consists of regular
scalars avoiding every $N_I$. Such a scalar acts surjectively on
every summand and hence on $E$. If an ideal meeting $S$ belonged
to $\Gamma$, it would be contained in its chosen $N_I$, a contradiction.
Both conditions of the proposition hold. This is the sdf version of
the construction in \cite[Example~2.5]{AhmedDumitrescu2017}.
\end{example}

An $A$-module $E$ is a \emph{multiplication module} if every submodule
is $IE$ for an ideal $I$ of $A$. Strong factorization imposes this
condition on every submodule.

\begin{proposition}\label{prop:all-homogeneous}
If $A\ltimes E$ is SSAF, then every ideal of $A\ltimes E$ is
homogeneous and $E$ is a multiplication $A$-module.
\end{proposition}
\begin{proof}
By Corollary~\ref{cor:ssaf-ssp}, every ideal is a product of radical
sdf ideals. Each such factor contains the square-zero ideal
$0\ltimes E$ and is homogeneous, so the product is homogeneous.
For a submodule $F$ factor $0\ltimes F=\prod_{i=1}^n(I_i\ltimes E)$
using radical sdf factors. Multiplication gives
\[
 F=\left(\sum_{j=1}^n\prod_{i\ne j}I_i\right)E.
\]
Here an empty product is $A$. Thus each submodule is $JE$ for an ideal
$J$, which is the definition of a multiplication module.
\end{proof}

\subsection{Von Neumann regular bases}

For a von Neumann regular base ring, the characteristic-$2$ part and
the remaining residue fields give a complete description of the
strong factorization property. In characteristic $2$, the criterion
reduces to the multiplication-module condition; the other residue
characteristics impose an additional finiteness requirement.
For general idealization background see \cite{AndersonWinders2009}.

We will use the following elementary observations.  Every ideal of a
commutative von Neumann regular ring is radical, and $IJ=I\cap J$ for
all ideals $I,J$.  Indeed, if $x\in I\cap J$, choose $b$ with
$x=x^2b$; then $x=x(xb)\in IJ$.  Every nonzero sdf-absorbing ideal of
an arbitrary ring is radical, but this assertion does not extend to
the zero ideal in an arbitrary ring.  The regular base rings below
are reduced, so every one of their ideals, including zero, is radical.

The following is the von Neumann regular specialization of
Hizem's minimal-prime criterion
\cite[Proposition~4 and Corollary~7]{Hizem2026}, including zero
because the coefficient ring is reduced.

\begin{lemma}\label{lem:vnr-sdf}
Let $A$ be a commutative von Neumann regular ring and $J$ a proper
ideal of $A$.  Then $J$ is sdf-absorbing if and only if at most one
maximal ideal containing $J$ avoids $2$.
Moreover, whenever $J$ is sdf-absorbing, its square-difference
condition holds for all $a,b\in A$, including zero elements.
\end{lemma}
\begin{proof}
The ring $A/J$ is reduced and zero-dimensional.  If at most one of
its residue fields has characteristic different from $2$, then
$a^2=b^2$ in $A/J$ implies that one of $a+b,a-b$ vanishes in all
residue fields, hence in $A/J$ itself.  If two maximal ideals avoid
$2$, choose an idempotent $e$ of $A/J$ taking the values $0$ and $1$
at those two ideals.  The elements $2e-1$ and $1$ have equal squares,
while their sum and difference are both nonzero.  Their lifts to $A$
are nonzero, which contradicts the original sdf condition.
Finally, if one of $a,b$ is zero and $a^2-b^2\in J$, radicality gives
the required conclusion.
\end{proof}

\begin{theorem}\label{thm:vnr-ssaf}
Let $A$ be a nonzero commutative von Neumann regular ring.  The
following are equivalent:
\begin{enumerate}
\item $A$ is an SSAF-ring.
\item Only finitely many maximal ideals of $A$ avoid $2$.
\item $A[1/2]$ is a finite product of fields, allowing the zero ring
as the empty product.
\item There is an isomorphism
\[
 A\cong C\times F_1\times\cdots\times F_r,
\]
where $C$ is a von Neumann regular ring with $2C=0$, possibly the
zero ring, and each $F_i$ is a field of characteristic different
from $2$.
\end{enumerate}
\end{theorem}
\begin{proof}
Suppose that $A$ is SSAF and factor its zero ideal as
$0=J_1\cdots J_n$ with the $J_i$ proper sdf-absorbing ideals.
For any maximal ideal $M$, primality implies that some $J_i$ is
contained in $M$.  By Lemma~\ref{lem:vnr-sdf}, each $J_i$ is contained
in at most one maximal ideal avoiding $2$.  Thus only finitely many
maximal ideals avoid $2$.

Since $A$ is von Neumann regular, the principal ideal $2A$ is generated
by an idempotent $e$.  Therefore
$A\cong(1-e)A\times eA$.  One has $2(1-e)A=0$, and $2$ is a unit in
$eA$, which is naturally isomorphic to $A[1/2]$.  Its maximal ideals
correspond exactly to the maximal ideals of $A$ avoiding $2$.
A reduced zero-dimensional ring with finitely many maximal ideals
is their residue-field product, by the Chinese remainder theorem.
This proves the equivalence of the last three conditions, including
the case $e=0$.

Under the displayed decomposition, write an arbitrary proper ideal
as $J=J_C\times J_1\times\cdots\times J_r$, where each $J_i$ is
either $0$ or $F_i$.  The ideal
$J_C\times F_1\times\cdots\times F_r$, when proper, is sdf-absorbing
because its quotient is reduced of characteristic $2$.
For each $i$ with $J_i=0$, the ideal with zero $i$th coordinate and
full other coordinates is prime.  The product of these ideals is
$J$, with any full-ring factor omitted.  Thus $A$ is SSAF.
\end{proof}

The next lemma proves homogeneity over a regular base.

The following homogeneity argument is valid for every module over a
von Neumann regular base.  It does not require finite generation or
faithfulness of the module.

\begin{lemma}\label{lem:vnr-homogeneous}
Let $A$ be a commutative von Neumann regular ring and $E$ an
$A$-module.  Every ideal of $A\ltimes E$ has the form $I\ltimes F$,
where $I$ is an ideal of $A$, $F$ is a submodule of $E$, and
$IE\subseteq F$.
The radical ideals are precisely the ideals $I\ltimes E$.
\end{lemma}
\begin{proof}
Let $H$ be an ideal and $(a,u)\in H$.  Choose $b\in A$ with
$a=a^2b$, and put $p=ab$.  Then $p^2=p$.  The element
\[
 (a,u)^2(b,0)-(a,u)=(0,(2p-1)u)
\]
belongs to $H$.  Since $(2p-1)^2=1$, multiplication by $(2p-1,0)$
shows that $(0,u)\in H$, and hence $(a,0)\in H$.
Thus $H$ is homogeneous, with the indicated containment condition.
The nilradical of $A\ltimes E$ is $0\ltimes E$, because $A$ is
reduced.  Any radical ideal must therefore contain $0\ltimes E$;
conversely, every quotient by $I\ltimes E$ is the reduced ring $A/I$.
\end{proof}

\begin{theorem}\label{thm:vnr-idealization-ssaf}
Let $A$ be a nonzero commutative von Neumann regular ring and $E$
an $A$-module.  Then $A\ltimes E$ is an SSAF-ring if and only if
$E$ is a multiplication $A$-module and only finitely many maximal
ideals of $A$ avoid $2$.
Equivalently, write
$A\cong C\times F_1\times\cdots\times F_r$ as in
Theorem~\ref{thm:vnr-ssaf}.  Then the conditions on $E$ are precisely
that its $C$-component is a multiplication $C$-module and that its
$F_i$-component has dimension at most one for every $i$.
\end{theorem}
\begin{proof}
Suppose first that $A\ltimes E$ is SSAF.  Its quotient $A$ is SSAF:
to factor a proper ideal $J$ of $A$, factor its inverse image
$J\ltimes E$ and then pass to the quotient by $0\ltimes E$.
Every factor contains $0\ltimes E$, so the sdf condition passes
directly to its image.  Theorem~\ref{thm:vnr-ssaf} now gives the
finiteness condition, and
Proposition~\ref{prop:all-homogeneous} shows that $E$
is a multiplication module.

Conversely, assume these conditions, and write
$A=C\times F_1\times\cdots\times F_r$.
The corresponding central idempotents decompose
$E=E_C\oplus E_1\oplus\cdots\oplus E_r$.
A module over a finite product is multiplication exactly when all
its components are multiplication modules.  A vector space over a
field is a multiplication module exactly when its dimension is at
most one.

Consider first $C\ltimes E_C$.  By
Lemma~\ref{lem:vnr-homogeneous}, a proper ideal is $I\ltimes F$.
Write $F=JE_C$, and replace $J$ by $I+J$; then $I\subseteq J$ and
$F=JE_C$ still holds.  Since $C$ is von Neumann regular, $IJ=I$.
Consequently
\[
 I\ltimes F=(I\ltimes E_C)(J\ltimes E_C).
\]
Every proper factor in this expression has a reduced quotient of
characteristic $2$, and hence is sdf-absorbing.  Full-ring factors
are omitted.  Thus $C\ltimes E_C$ is SSAF by factorizations into
radical sdf ideals.

For each $i$, the ring $F_i\ltimes E_i$ is either $F_i$ or isomorphic
to $F_i[\varepsilon_i]/(\varepsilon_i^2)$.  In the latter ring the
only proper ideals are $0$ and the prime ideal $(\varepsilon_i)$,
and $0=(\varepsilon_i)^2$.  Thus these components also factor all
proper ideals into prime ideals.

Finally,
\[
 A\ltimes E\cong(C\ltimes E_C)\times
                    \prod_{i=1}^r(F_i\ltimes E_i).
\]
Lift the factors just constructed from each component to this finite
product and multiply them.  Every lifted factor is sdf-absorbing:
all the constructed factors are radical, so their square-difference
conditions hold even when one coordinate input is zero.  This proves the required property for the finite product.
\end{proof}

The case of a module containing a faithful cyclic submodule has an
especially rigid form. The underlying SSP and almost-multiplication
criterion is \cite[Corollary~2.4]{DumitrescuMahdouZahir2021}.

\begin{corollary}\label{cor:faithful-element}
Let $A$ be a nonzero ring and let $E$ contain an element $e$ with
$\Ann_A(e)=0$. Then $A\ltimes E$ is SSAF if and only if $A$ is
von Neumann regular, $E\cong A$, and only finitely many maximal ideals
of $A$ avoid $2$. In characteristic $2$, this is further equivalent
to $A\ltimes E$ being SSP, and to its being an almost multiplication ring.
\end{corollary}
\begin{proof}
By \cite[Corollary~2.4]{DumitrescuMahdouZahir2021}, under the stated
hypothesis the SSP and almost-multiplication conditions are equivalent
to $A$ being von Neumann regular and $E\cong A$.
SSAF implies SSP, and Theorem~\ref{thm:vnr-idealization-ssaf} then
adds exactly the finite odd-spectrum condition. Conversely, that
same theorem proves SSAF because $A$ is a multiplication module over
itself. In characteristic $2$ there are no maximal ideals avoiding $2$.
\end{proof}

\begin{example}
Let $B$ be SSAF, let $C$ be von Neumann regular with $2C=0$, and let
$E$ be a multiplication $C$-module. Regard $E$ as a $B\times C$-module
through the second projection. Then
$(B\times C)\ltimes E\cong B\times(C\ltimes E)$ is SSAF by
Theorem~\ref{thm:vnr-idealization-ssaf} and
Proposition~\ref{prop:permanence}. This permits an arbitrary SSAF
component alongside a square-zero component.
\end{example}

\subsection{Comparison with absorbing factorizations}

The SSAF criterion can now be compared with ordinary absorbing
factorization over the same base. For $n\geq1$, a proper ideal $I$ is \emph{$n$-absorbing} if a product
of $n+1$ ring elements that belongs to $I$ always has an $n$-element subproduct
in $I$. A ring is \emph{$n$-AF} if every proper ideal factors into
$n$-absorbing ideals, and $\operatorname{AF\!-\!dim}(R)$ is the least
such $n$, or infinity when none exists. TAF means $2$-AF.
A proper ideal is \emph{$n$-OA} if
$a_1\cdots a_{n+1}\in I$ for nonunits $a_i$ implies
$a_1\cdots a_n\in I$ or $a_{n+1}\in I$; $n$-OAF requires
factorization of every proper ideal into such ideals. The following comparison
combines finite-prime-support results with the preceding
idealization theorem.

\begin{theorem}\label{thm:comparison-vnr-idealizations}
Let $A$ be a nonzero commutative von Neumann regular ring and let
$E$ be any $A$-module.  Put $R=A\ltimes E$.
\begin{enumerate}
\item $R$ is TAF if and only if $A$ is a finite product of fields.
\item Its absorbing-factorization dimension is
\[
 \operatorname{AF\!-\!dim}(R)=
 \begin{cases}
 1,&|\Max(A)|<\infty\text{ and }E\text{ is a multiplication module},\\
 2,&|\Max(A)|<\infty\text{ and }E\text{ is not a multiplication module},\\
 \infty,&\operatorname{Max}(A)\text{ is infinite}.
 \end{cases}
\]
\item Fix $n\geq2$.  The ring $R$ is $n$-OAF if and only if
either $A$ is a field, or $A$ is a finite product of at least two
fields and $E$ is a multiplication $A$-module.
\item The ring $R$ is both SSAF and TAF if and only if it is ZPI;
equivalently, $A$ is a finite product of fields and $E$ is a
multiplication module.
\end{enumerate}
\end{theorem}
\begin{proof}
If $R$ has finite AF-dimension, its quotient $A$ has finite
AF-dimension.  Every ideal of such a ring has finitely many
minimal primes; see
\cite[Proposition~3.7]{AhmedDumitrescuKhadam2020}.
Since $A$ is von Neumann regular, its minimal primes are its
maximal ideals, so $A$ is a finite product of fields.
For the TAF case this also follows directly from
\cite[Propositions~2.2 and 2.3]{MukhtarAhmedDumitrescu2018}.

Write $A=F_1\times\cdots\times F_r$ and decompose
$E=E_1\oplus\cdots\oplus E_r$ accordingly.
Then $R\cong\prod_i(F_i\ltimes E_i)$.
In each component the maximal ideal is $0\ltimes E_i$ and has
square zero.  Every proper ideal is $2$-absorbing: if at least
two of three testing elements are nonunits, their product is
zero; otherwise a unit can be cancelled to put a two-factor
subproduct in the ideal.
Finite products of TAF-rings are TAF by
\cite[Proposition~2.4]{MukhtarAhmedDumitrescu2018}.

The component $F_i\ltimes E_i$ is ZPI exactly when
$\dim_{F_i}E_i\leq1$.  Indeed its only prime ideal is
$0\ltimes E_i$, whose square is zero; a proper nonzero subspace
of $E_i$ cannot be the second coordinate of a prime-ideal
product.  If the dimension is at most one, the only proper
ideals are zero and that maximal ideal, and zero is its square
when $E_i\neq0$.  Componentwise dimension at most one is
equivalent to $E$ being a multiplication module.
These observations prove (1) and (2), using finite-product
behavior or directly factoring the coordinate ideals.

If $A$ is a field and $n\geq2$, every proper ideal of
$A\ltimes E$ is $n$-OA, since the product of the first $n$
testing nonunits is already zero.  If $A$ is not a field,
$R$ is nonlocal.  Lemma~2.1(1) of
\cite{ElKhalfiLaarabiKoc2026} says that its $n$-OA ideals are
exactly its prime ideals, so $R$ is $n$-OAF exactly when it is
ZPI.  Quotient inheritance first forces $A$ to be a finite
product of fields.  This proves (3).

Finally, Theorem~\ref{thm:vnr-idealization-ssaf} requires a multiplication module and only
finitely many maximal ideals of $A$ avoiding $2$.  Combining
that criterion with (1) gives (4).
\end{proof}

\begin{example}
For every field $k$, the ring $k\ltimes k^2$ is TAF and is
$n$-OAF for every $n\geq2$, but is not SSAF.  The failure of
the multiplication-module condition is witnessed by any
one-dimensional proper subspace of $k^2$.
In the opposite direction, an infinite Boolean ring is SSAF
because every proper ideal is radical and $2=0$, but has
infinite AF-dimension because its zero ideal has infinitely
many minimal primes.
\end{example}

\subsection{Concluding remarks}

For domains, the extra restriction imposed by square-difference
factorization is concentrated at maximal ideals avoiding $2$.
The torsion decomposition explains the same distinction at the
module level, while complementary overrings separate invertible
ideals before passage to ideal classes. In rings with zero divisors,
regular ideals can fail to detect large parts of the spectrum;
idealizations make this failure explicit. The strong property
recovers that information and, over a von Neumann regular base,
combines a finite condition on the odd residue fields with the
multiplication-module condition.

\raggedbottom

\end{document}